\documentclass[12pt]{amsart} 
\makeatletter
\providecommand{\@LN}[2]{}
\makeatother
\calclayout

\usepackage{microtype}
\usepackage[english]{babel}
\usepackage[maxbibnames=5,minbibnames=4,backend=bibtex]{biblatex}
\usepackage{csquotes}
\usepackage{amssymb}
\usepackage{amsmath}
\usepackage{amsthm}
\newtheorem{thm}{Theorem}[section]
\newtheorem{lem}[thm]{Lemma}
\newtheorem{prop}[thm]{Proposition}
\newtheorem{coro}[thm]{Corollary}

\theoremstyle{remark}
\newtheorem{rema}[thm]{Remark}
\newtheorem{exa}[thm]{Example}
\newtheorem{defi}[thm]{Definition}

\usepackage{MnSymbol}
\usepackage{tikz}
\usepackage[all,cmtip]{xy}
\usepackage{todonotes}
\usepackage{stmaryrd}   
\title{Reduction of symplectic groupoids and quotients of quasi-Poisson manifolds}
\author{Daniel \'Alvarez}\address{Instituto de Matem\'atica Pura e Aplicada, Estrada Dona Castorina, 110, Jardim Bot\^anico, CEP \texttt{22460320}, Rio de Janeiro, Brasil}
\email{uerbum@impa.br}
\date{} 
\begin{document}


\begin{abstract} In this work we study the integrability of quotients of quasi-Poisson manifolds. Our approach allows us to put several classical results about the integrability of Poisson quotients in a common framework. By categorifying one of the already known methods of reducing symplectic groupoids we also describe double symplectic groupoids which integrate the recently introduced Poisson groupoid structures on gauge groupoids.
\end{abstract}\maketitle
\tableofcontents 

\section{Introduction} 

One of the main tools in the study of Poisson manifolds is the concept of symplectic groupoid, see \cite{catfel,groqua,cradifcoh,conlin,craruimar} for some of its most exciting applications. So a basic problem in Poisson geometry is the construction of interesting examples of symplectic groupoids. Unlike finite-dimensional Lie algebras, which always admit integrations to Lie groups, not every Poisson manifold is ``integrable'' to a symplectic groupoid \cite{weisygr}. Although general criteria for the integrability of Poisson manifolds (and Lie algebroids in general) were established in \cite{crarui,craruipoi}, in most cases, these conditions do not easily lead to an explicit construction. 

In this paper we address the problem of describing symplectic groupoids which integrate Poisson manifolds obtained as quotients of Lie groupoid actions on q-Poisson manifolds. The study of the integrability of quotient Poisson structures began with \cite{mikwei}, where it was established that  the quotient of a symplectic manifold $S$ by a Lie group action is integrable by performing Marsden-Weinstein reduction on the fundamental groupoid of $S$. Subsequently, it was proven in \cite{mommap} that the quotient of an integrable Poisson manifold $S$ by a Lie group action by automorphisms is also integrable by a Marsden-Weinstein quotient of the source-simply-connected integration of $S$. The work in \cite{ferigl,morla} generalized this result for Poisson actions of Poisson groupoids on integrable Poisson manifolds. In this work, we generalize these results even further by considering Poisson quotients of {\em quasi-Poisson (q-Poisson) manifolds}

The study of q-Poisson manifolds began with the finite-dimensional description of Poisson structures on representation varieties provided in \cite{alemeikos}. Later on, it was realized that most of the known methods of Poisson reduction by symmetries could be described in terms of an even broader notion of q-Poisson manifold \cite{quapoi}. Roughly, a {\em q-Poisson manifold} (in the general sense of \cite{quapoi}) is a manifold endowed with a suitable (global or infinitesimal) action and a bilinear bracket on the space of smooth functions that fails to be Poisson in a way controlled by the action; an important feature is that the orbit space of such an action turns out to carry a genuine Poisson structure, provided it is smooth. Let $(S,\pi)$ be a q-Poisson manifold for a Lie quasi-bialgebroid $(A,\delta,\chi)$ such that the moment map $J:S \rightarrow M$ is a surjective submersion. There is a canonical Lie algebroid structure on the conormal bundle $C$ of the $A$-orbits, provided it is smooth. If the $A$-action on $S$ integrates to a $G$-action, where $G \rightrightarrows M $ is a Lie groupoid integrating $A$, and if the $G$-action on $S$ is free and proper, then the quotient $S/G$ inherits a canonical Poisson structure $\sigma$. 
\begin{thm}\label{main} The Poisson manifold $(S/G,\sigma)$ is integrable if and only if the Lie algebroid $C$ is integrable. Moreover, if $\mathcal{G} (C)$ is the source-simply-connected integration of $C$, then there is a lifted $G$-action on $\mathcal{G} (C)$ such that the orbit space $\mathcal{G} (C)/G $ is a symplectic groupoid integrating $(S/G,\sigma)$. \end{thm} 

Theorem \ref{main} is a consequence of the integrability of Lie groupoid actions on Lie algebroids established in \cite{intinfact}. We obtain some corollaries about Poisson reduction such as the following. Let $G$ be a Poisson group acting freely and properly by a Poisson action on a Poisson manifold $(S,\pi)$. Then the induced Poisson structure on $\overline{S}=S/G$ is integrable if $S$ is integrable \cite{mommap,ferigl}. For the original q-Poisson $\mathfrak{g}$-manifolds of \cite{alemeikos} we get a completely analogous result. If $(S,\pi)$ is a q-Poisson $G$-manifold, then there is a nonobvious but canonical Lie algebroid structure on $T^*S$ \cite{liblasev}. Theorem \ref{main} implies that the Poisson structure induced on $\overline{S}= S/G$ is integrable if $T^*S$ is integrable. In both of these situations, the Lie algebroid $C$ that appears in Theorem \ref{main} controlling the integrability of the quotient can be interpreted as the Lie algebroid of the level set corresponding to the unit of a Lie group valued moment map as in \cite{luphd} in the former case and in the sense of \cite{alemeimal} in the latter. This last observation is related to the integration of Poisson structures on moduli spaces of flat $G$-bundles that shall be studied in a companion paper, see \cite{intquopoi}. 

Poisson actions also allow us to obtain Poisson quotients by considering the action restricted to a coisotropic subgroup \cite{semdres,luphd}. In this context, the integrability of the quotient can be proved only under the assumption that the acting Poisson group is complete \cite{ferigl}. In this context, we have the following result which can be seen as a simple categorification of the quotient construction of  \cite{ferigl}.
\begin{thm} Let $G$ be a complete Poisson group acting freely and properly on a Poisson manifold $M$. If $M$ is integrable, then the gauge Poisson groupoid $(M \times \overline{M})/G \rightrightarrows M/G$ is integrable by a double symplectic groupoid. \end{thm} 
This last result, together with an observation coming from \cite{folpoigro} about the symplectic leaves of Poisson groupoids, can be applied to an interesting family of examples of gauge Poisson groupoids recently introduced by J.-H. Lu and her collaborators thereby producing a number of new examples of symplectic groupoids.
 
\subsection*{Acknowledgments.} The author thanks CNPq for the financial support and H. Bursztyn for his continuous advice and support. 

\section{Preliminaries} \subsection{Lie groupoids and Lie algebroids} A {\em smooth groupoid} $G$ over a manifold $M$, denoted $G \rightrightarrows M$, is a groupoid object in the category of not necessarily Hausdorff smooth manifolds such that its source map is a submersion. The structure maps of a groupoid are its source, target, multiplication, unit map and inversion, denoted respectively $\mathtt{s},\mathtt{t},\mathtt{m},\mathtt{u} , \mathtt{i}$. For the sake of brevity, we also denote $\mathtt{m}(a,b)$ by $ab$. A {\em Lie groupoid} is a smooth groupoid such that its base and source-fibers are Hausdorff manifolds, see \cite{moeint,dufzun}. A Lie groupoid is \emph{source-simply-connected} if its source fibres are 1-connected.  

\emph{A left Lie groupoid action} of a Lie groupoid $G \rightrightarrows M$ on a map $J:S \rightarrow M$ is a smooth map $a:G_\mathtt{s} \times_J S \rightarrow S$ such that (1) $a(\mathtt{m} (g,h),x)=a(g,a(h,x))$ for all $g,h\in G$ and for all $x\in S$ for which $a$ and $\mathtt{m}$ are defined and (2) $a(\mathtt{u}(J(x)),x)= x$ for all $x\in M$, the fiber product $G_\mathtt{t} \times_J G $ is denoted by $G \times_M S$. There is a Lie groupoid structure on $G \times_M S$ over $S$ with the projection $\text{pr}_2:G \times_M S \rightarrow S$  being the source map, $a$ being the target map and the multiplication given by $(g,a(h,p))(h,p)=(gh,p)$. The Lie groupoid $G \times_M S\rightrightarrows S$ thus obtained is called an {\em action groupoid}. A Lie groupoid action as before is free if the associated action groupoid has trivial isotropy groups; a Lie groupoid action is proper if the map $(a,\text{pr}_2): G \times_M S \rightarrow S \times S$ is proper. If a Lie groupoid is a Lie group, this recovers the usual notion of free and proper actions. 

A vector bundle $A$ over a manifold $M$ is a {\em Lie algebroid} if there exist: (1) a bundle map $\mathtt{a}:A\rightarrow TM$ called the {\em anchor} and (2) a Lie algebra structure $[\,,\,]$ on $\Gamma(A)$ such that the Leibniz rule holds
\[ [u,fv]=f[u,v]+\left(\mathcal{L}_{\mathtt{a}(u)}f\right)v, \] 
for all $u,v\in \Gamma(A)$ and $f\in C^\infty(M)$. See \cite{higmacalg,vai} for the definition of Lie algebroid morphism.

The {\em Lie algebroid} $A=A_{G}$ of a Lie groupoid $G\rightrightarrows B$ is the vector bundle $A=\ker T\mathtt{s}|_{B}$ endowed with the restriction of $T\mathtt{t}$ to $A$ as the anchor and with the bracket defined by means of right invariant vector fields \cite{moeint,macgen}. A Lie groupoid morphism induces a Lie algebroid morphism between the associated Lie algebroids, this construction defines a functor called the {\em Lie functor} that we denote by $\text{Lie} $. A Lie algebroid  which is isomorphic to the tangent Lie algebroid of a Lie groupoid is called {\em integrable}. If $A$ is an integrable Lie algebroid, we denote by $\mathcal{G} (A)$ its source-simply-connected integration (which is unique up to isomorphism). 

A fundamental result relating Lie groupoids and Lie algebroids is {\em Lie's second theorem}. Let $\phi:A \rightarrow B$ be a Lie algebroid morphism between integrable Lie algebroids. Then, for every Lie groupoid $K$ integrating $B$, there exists a unique Lie groupoid morphism $\Phi:\mathcal{G} (A) \rightarrow K$ such that $\text{Lie} (\Phi)=\phi$ \cite{intinfact,moeint}.
 

\subsection{Poisson structures, closed IM 2-forms and symplectic groupoids}\label{subsec:poi} A {\em Poisson structure} on a manifold $M$ is a bivector field $\pi\in \Gamma(\wedge^2TM)$ such that $[\pi,\pi]=0$, where $[\,,\,]$ is the Schouten bracket; with this kind of structure, $(M,\pi)$ is called a \emph{Poisson manifold}. There is a canonical Lie algebroid structure on the cotangent bundle of a Poisson manifold in which the Lie bracket on $\Omega^1(M)$ is given by
\begin{align*}  [\alpha,\beta]=\mathcal{L}_{\pi^\sharp(\alpha)}\beta-\mathcal{L}_{\pi^\sharp(\beta)}\alpha-d\pi(\alpha,\beta),  \end{align*}
for all $\alpha,\beta \in \Omega^1(M)$, and the anchor is the map $\pi^\sharp$ defined by $\alpha \mapsto i_\alpha \pi$; $T^*M$ is usually called the {\em cotangent Lie algebroid} of $M$. A {\em Poisson morphism} $J:(P,\pi_P) \rightarrow (Q,\pi_Q)$ between Poisson manifolds is a smooth map that satisfies $\pi_P \sim_J \pi_Q$.

Let $A$ be a Lie algebroid over $M$. A {\em closed IM 2-form on $A$} \cite{burint} is a vector bundle morphism $\mu:A \rightarrow T^*M$ over the identity such that
\begin{align} &\langle \mu (v ) ,\mathtt{a}( u) \rangle =-\langle \mu (u ),\mathtt{a}(v)   \rangle \label{eq:im21} \\
&\mu([u,v])=\mathcal{L}_{\mathtt{a}  (u )} \mu(v) -i_{\mathtt{a} (v)} d \mu(u) \label{eq:im22}  \end{align}
for all $u,v \in \Gamma (A)$.  In the case of a Poisson manifold $(M,\pi)$, the identity on $T^*M$ is a closed IM 2-form. Closed IM 2-forms are the basic infrastructure necessary for performing Poisson reduction as we shall see below, see \cite{quoim2} for a general discussion. At the Lie groupoid level, a closed IM 2-form induces a {\em closed multiplicative 2-form}: let $G \rightrightarrows M$ be a Lie groupoid and take $\omega \in \Omega^2(G)$, $\omega $ is called multiplicative if $\text{pr}_1^* \omega +\text{pr}_2^* \omega -\text{pr}_3^* \omega $ vanishes on the graph of the multiplication inside $G\times G\times {G}$. A Lie groupoid $G \rightrightarrows B$ is a \emph{symplectic groupoid} \cite{weisygr,karpoi} if it is endowed with a symplectic form which is multiplicative. If the cotangent Lie algebroid of a Poisson manifold $M$ is integrable, we shall say that $M$ is {\em integrable}. If $M$ is an integrable Poisson manifold, we denote by $\Sigma(M) \rightrightarrows M$ its source-simply-connected integration which naturally becomes a symplectic groupoid by integrating its canonical closed IM 2-form \cite{burint}.

\section{Integrability of quotients of quasi-Poisson manifolds}\label{sec:qpoi} The integrability of Poisson manifolds obtained by reduction has been studied in \cite{mikwei,mommap,morla,ferigl}. In this section we put some of the results contained in those works in the broader context of q-Poisson manifolds.
\subsubsection{Poisson structures and quasi-Poisson manifolds} 
\begin{defi}[\cite{roycou}] A {\em Lie quasi-bialgebroid} is a Lie algebroid $A$ over $M$ endowed with a degree one derivation $\delta : \Gamma (\wedge^k A) \rightarrow \Gamma ( \wedge^{k+1} A)$ for all $k\in \mathbb{N} $ which is a derivation of the bracket on $A$, 
\[ \delta([u,v])=[\delta(u),v]+(-1)^{p-1}[u,\delta(v)]\] 
for all $u\in \Gamma(\wedge^p A) $, $v\in \Gamma(\wedge^\bullet A) $ and satisfies $\delta^2=[\chi,\,]$, where $\chi\in \Gamma (\wedge^3 A)$ is such that $\delta(\chi)=0$. \end{defi}

Since $\delta$ is a derivation, it is determined by its restriction to degree 0 and degree 1 where it is given respectively by a vector bundle map $\mathtt{a}_*:A^* \rightarrow TM$ and a map $\Gamma ( A) \rightarrow \Gamma ( \wedge^{2} A)$ called the {\em cobracket}. 

\begin{exa} A {\em Lie bialgebroid} $(A,A^*)$ is a Lie quasi-bialgebroid $(A,\delta,\chi)$ in which $\chi=0$ and hence the differential $\delta$ satisfies $\delta^2=0$ \cite{macxu}. Since the dual of a differential which squares to zero is a Lie bracket, a Lie bialgebroid consists of a pair of Lie algebroid structures on $A$ and $A^*$ which are compatible in a suitable sense. A {\em Lie bialgebra} $(\mathfrak{g},\mathfrak{g}^*)$ is a Lie bialgebroid over a point \cite{driham}. \end{exa}
   
\begin{defi}[\cite{quapoi}]\label{def:qpoi} A {\em quasi-Poisson manifold} (or a {\em Hamiltonian space}) for a Lie quasi-bialgebroid $(A,\delta,\chi)$ on $M$ is given by an action $\rho:J^*A \rightarrow TS$ of $A$ on a smooth map $J:S \rightarrow M$ and a bivector field $\pi$ on $S$ such that:
 \begin{align*} & \frac{1}{2}[\pi,\pi]=\rho(\chi) \\
	&  \mathcal{L}_{\rho(U)}\pi=\rho(\delta(U)), \quad \forall U\in \Gamma (A), \\
&\pi^\sharp J^*=\rho\circ \mathtt{a}^*_*, \end{align*}
			where $\mathtt{a}_*:A^* \rightarrow TS $ is the component of $\delta$ in degree zero as before. \end{defi} 
\begin{exa} An infinitesimal Poisson action $\rho:\mathfrak{g} \rightarrow TS$ of the tangent Lie bialgebra $(\mathfrak{g} ,\mathfrak{g}^*)$ of a Poisson group \cite{driham,luphd} can be expressed by saying that $(S,\pi,\rho)$ is a Hamiltonian space for $(\mathfrak{g},\delta,0)$, where $\delta$ is the differential dual to the bracket on $\mathfrak{g}^*$. \end{exa}  
\begin{exa} The original q-Poisson manifolds, which were introduced in \cite{alemeikos}, are a special case of Definition \ref{def:qpoi}. Consider a Lie algebra $\mathfrak{g}$ endowed with an Ad-invariant symmetric nondegenerate bilinear form $B$. Then there is a Lie quasi-bialgebra structure $(\mathfrak{g} ,\delta, \chi) $ on $\mathfrak{g} $ which depends on $B$ and is determined by the splitting of $\mathfrak{g} \oplus \mathfrak{g}$ as the sum of the diagonal Lie subalgebra and the anti-diagonal \cite{alekos,alemeikos}. The q-Poisson manifolds corresponding to $(\mathfrak{g} ,\delta, \chi) $ shall be called {\em q-Poisson $\mathfrak{g}$-manifolds} in accordance to \cite{alemeikos,liblasev}. \end{exa} 
The main feature of this notion is the following well known reduction construction. Let $(S,\pi)$ be a quasi-Poisson manifold for a Lie quasi-bialgebroid $(A ,\delta, \chi) $. If the $A$-action induces a simple foliation on $S$, then its leaf space $\overline{S} $ inherits a unique Poisson structure $\overline{\pi}$ such that $\pi$ and $\overline{\pi}$ are $q$-related, where $q:S \rightarrow \overline{S}$ is the projection. In fact, the cotangent Lie algebroid of $\overline{S}$ can be described as the quotient of a Lie algebroid over $S$.     
\begin{prop}\label{pro:con} If the $A$-action induces a regular foliation on $S$, then the conormal bundle $C$ of the $A$-orbits is a Lie algebroid with the anchor defined by $\alpha \mapsto \pi^\sharp(\alpha )$ and the Lie bracket 
\begin{align} [\alpha ,\beta ]_C:=\mathcal{L}_{\pi^\sharp(\alpha) }\beta -i_{\pi^\sharp(\beta )}d\alpha, \label{eq:bra} \end{align} 
for all $\alpha ,\beta \in \Gamma (C)$. \end{prop} 
\begin{proof} Take $U\in \Gamma (B)$, $ \alpha \in \Gamma (C)$ and let $\beta \in \Omega^1(S)$ be arbitrary. Then we have that 
\[ \langle \beta ,\mathcal{L}_{\rho(U)}(\pi^\sharp(\alpha )) \rangle =\mathcal{L}_{\rho(U)}\langle \beta , \pi^\sharp (\alpha )\rangle - \langle \mathcal{L}_{\rho(U)} \beta , \pi^\sharp(\alpha ) \rangle =\langle \rho(\delta(U)),\alpha \wedge \beta \rangle + \langle \beta ,\pi^\sharp(\mathcal{L}_{\rho(U)} \alpha ) \rangle,  \] 
where we used the identity $\mathcal{L}_{\rho(U)} \pi= \rho (\delta(U))$ and the fact that
\[ \mathcal{L}_{\rho(U)} (\pi( \alpha , \beta ))=(\mathcal{L}_{\rho(U)} \pi)(\alpha , \beta )+\pi (\mathcal{L}_{\rho(U)} \alpha , \beta )+\pi ( \alpha ,\mathcal{L}_{\rho(U)} \beta ). \]
But $\langle \rho(\delta(U)),\alpha \wedge \beta \rangle=0$ since $\alpha $ lies in the annihilator of $\rho(\Gamma (A))$. As a consequence, 
\begin{align}  \mathcal{L}_{\rho(U)}\pi^\sharp(\alpha )=\pi^\sharp(\mathcal{L}_{\rho(U)}\alpha ). \label{eq:infact0} \end{align} 
Now we shall check that $[\alpha ,\beta ]_C\in \Gamma (C)$. Take $U\in \Gamma (A)$, we have that 
\[ \langle \mathcal{L}_{\pi^\sharp(\alpha) }\beta , \rho(U) \rangle =\langle d ( i_{\pi^\sharp(\alpha) } \beta) + i_{\pi^\sharp(\alpha) } d \beta ,\rho(U) \rangle =\langle \beta ,[\rho(U),\pi^\sharp(\alpha )]\rangle=-\langle \mathcal{L}_{\rho(U)} \alpha ,\pi^\sharp(\beta ) \rangle,  \]
where we used \eqref{eq:infact0} in the last equality. On the other hand, 
\begin{align*} &\langle i_{\pi^\sharp(\beta )}d\alpha,\rho(U) \rangle =d \alpha (\pi^\sharp(\beta ),\rho(U))=- \mathcal{L}_{\rho(U)}\langle \alpha ,\pi^\sharp(\beta ) \rangle +\langle \alpha ,\mathcal{L}_{\rho(U)} \pi^\sharp(\beta ) \rangle. \end{align*}
By combining the last two equations we get that
\[ \langle [\alpha, \beta ]_C,\rho(U) \rangle =\langle \alpha, \mathcal{L}_{\rho(U)}(\pi^\sharp)(\beta ) \rangle=(\mathcal{L}_{\rho(U)}\pi)(\alpha ,\beta ) \]
and this last term is zero because $\mathcal{L}_{\rho(U)}\pi=\rho(\delta(U))$ is generated by vectors tangent to the $A$-orbits. Finally, 
\[ \pi^\sharp([\alpha ,\beta ]_C)=[\pi^\sharp (\alpha ), \pi^\sharp(\beta )]-\frac{1}{2} i_{\alpha \wedge \beta }[\pi,\pi]=[\pi^\sharp (\alpha ), \pi^\sharp(\beta )]- i_{\alpha \wedge \beta }\rho(\chi)=[\pi^\sharp (\alpha ), \pi^\sharp(\beta )]. \]
So it follows that $[\,,\,]_C$ satisfies the Jacobi identity and hence it endows $C$ with a Lie algebroid structure. \end{proof} 
\begin{rema} As we shall see next, the integrability of the quotient Poisson structure is controlled by the integrability of this Lie algebroid structure on $C$. It is immediate that the inclusion $C \hookrightarrow T^*S$ is a closed IM 2-form on $C$, so we see that quasi-Poisson reduction fits within the general framework of Poisson reduction described in \cite{quoim2}. \end{rema}  
A q-Poisson manifold $(S,\pi)$ for a Lie quasi-bialgebroid $(A,\delta,\chi)$ comes naturally equipped with a canonical action of $A$ on $C$ as in Proposition \ref{pro:con} in the following sense. 
  Let $J:S \rightarrow M$ be a surjective submersion, let $A$ be a Lie algebroid over $M$ and let $C$ be a Lie algebroid over $S$ such that $Tq\circ \mathtt{a}=0$, where $\mathtt{a}$ is the anchor of $C$. 
\begin{defi}[\cite{higmacalg,intinfact}]\label{def:infact} An {\em action of $A$ on a Lie algebroid $C$ over $S$} is an $A$-action on $J$ and a Lie algebra morphism $\Gamma (A) \rightarrow \text{Der}(C)$ which is $C^\infty(M)$-linear, where $\text{Der}(C)$ is the space of derivations of $C$. \end{defi}   
 \begin{lem}\label{leminfact} Let $(S,\pi)$ be a Hamiltonian space for a Lie quasi-bialgebroid $(A,\delta,\chi)$ on $M$ with moment map $J:S \rightarrow M$. If $J$ is a surjective submersion, then the map given by $U\mapsto \mathcal{L}_{\rho(U)}$ for all $U\in \Gamma (A)$ defines an infinitesimal action of $A$ on $C$. \end{lem}  
\begin{proof} First of all, equation \eqref{eq:infact0} implies that $\mathcal{L}_{\rho(U)}[\alpha ,\beta ]_C =[ \mathcal{L}_{\rho(U)}\alpha ,\beta ]_C+[\alpha ,\mathcal{L}_{\rho(U)} ,\beta ]_C$. On the other hand, the equation $\pi^\sharp J^*=\rho\circ \mathtt{a}^*_*$ implies that $TJ\circ \pi^\sharp:C \rightarrow TM$ is the zero map. Since $\mathcal{L}_{f\rho(U)} \alpha =f \mathcal{L}_{\rho(U)} \alpha $ by Cartan's formula, the map $\psi$ is $C^\infty(M)$-linear and so we are done. \end{proof}
We are only interested in the situation in which the previous infinitesimal action of $A$ on $C$ is integrable by a global action of the following kind. Let $G \rightrightarrows M$ be a Lie groupoid acting on a surjective submersion $J: S \rightarrow M$ and let $C$ be a Lie algebroid over $S$ such that $TJ\circ \mathtt{a}=0 $. In this situation we have an action of $C$ on the projection $G \times_M S \rightarrow S$ given by $X \mapsto (0,\mathtt{a}(X))$ for all $X\in \Gamma (C)$. Hence, there is an action Lie algebroid structure on $G \times_M C$ over $G \times_M S$. Let $p:C \rightarrow S$ be the vector bundle projection.
\begin{defi}[\cite{intinfact}]\label{def:infgroact} {\em An action of $G \rightrightarrows M$ on $C$} is a Lie groupoid action of $G$ on $J\circ p:C \rightarrow S$ such that the structure maps of the action groupoid $G \times_M C \rightrightarrows C$ are Lie algebroid morphisms over the structure maps of the action groupoid $G \times_M S \rightrightarrows S$. \end{defi}
\subsubsection{The integrability criterion} Let $(S,\pi)$ be a q-Poisson manifold for a Lie quasi-bialgebroid $(A,\delta,\chi)$ such that the moment map $J:S \rightarrow M$ is a surjective submersion and the $A$-action on $S$ induces a regular foliation. Suppose that the $A$-action on $C$ integrates to a $G$-action as in Definition \ref{def:infgroact} and that the $G$-action is free and proper so that the quotient $S/G$ is a smooth manifold and it inherits a Poisson structure $\sigma$ with the property that $Tq\circ \pi^\sharp\circ q^*=\sigma^\sharp$, where $q:S \rightarrow S/G$ is the quotient map.

\begin{thm}\label{thm:main} The Poisson manifold $(S/G,\sigma)$ is integrable if and only if the Lie algebroid $C$ is integrable. Moreover, if $\mathcal{G} (C)$ is the source-simply-connected integration of $C$, then there is a lifted $G$-action on $\mathcal{G} (C)$ such that the orbit space $\mathcal{G} (C)/G $ is a symplectic groupoid integrating $(S/G,\sigma)$. \end{thm}
\begin{rema} The previous result generalizes \cite[Thm 3.4.4]{morla}, which regards only Poisson groupoid actions. \end{rema} 
The $G$-action on $\mathcal{G} (C)$ as in the previous theorem is compatible with the groupoid structure in the following sense. 

Let $K \rightrightarrows S$ and $G \rightrightarrows M$ be Lie groupoids and let $J:S \rightarrow M$ be a surjective submersion such that $J\circ\mathtt{s}=J\circ \mathtt{t}$.
\begin{defi}[\cite{higmacalg}]\label{def:glogroact} An action of $G \rightrightarrows M$ on $K \rightrightarrows S$ is an action on the map $J\circ\mathtt{s}=J\circ \mathtt{t}:K \rightarrow M$ which is a Lie groupoid action $G \times_M K \rightarrow K$ such that it is a Lie groupoid morphism with respect to the fiber product groupoid $G \times_M K \rightrightarrows G \times_M S$. \end{defi}
Notice that when $M$ is a point, a $G$-action in the previous sense is a $G$-action by automorphisms on $K$.

\begin{proof}[Proof of Theorem \ref{thm:main}] Suppose that $C$ is integrable. 

{\em Step 1: lift of the $G$-action to $\mathcal{G} (C)$.} First of all, \cite[Thm. 3.6]{intinfact} implies that the $G$-action on $C$ lifts to a $G$-action on $\mathcal{G} (C)$. On the other hand, this lifted action is principal since it is principal on the base and hence the quotient $\mathcal{G} (C)/G$ inherits a unique Lie groupoid structure such that the projection map $\mathcal{G} (C) \rightarrow \mathcal{G} (C)/G$ is a Lie groupoid morphism  \cite[Lemma 2.1]{intinfact}. 

{\em Step 2: existence of a canonical multiplicative 2-form on $\mathcal{G} (C)$.} Recall that the inclusion $\mu: C \hookrightarrow T^*S$ constitutes a closed IM 2-form on $C$. Let $\lambda$ be the canonical 1-form on $T^*S$ and let us denote $\Lambda_\mu=d(\mu^*\lambda)\in \Omega^2(C)$. Then $\Lambda_\mu^\flat:TC \rightarrow T^*C$ is a Lie algebroid morphism which lifts to a Lie groupoid morphism $\omega^\flat:T \mathcal{G} (C) \rightarrow T^* \mathcal{G} (C)$ with the property that $\omega $ is a multiplicative closed 2-form on $\mathcal{G} (C)$ \cite{burcab}. 

{\em Step 3: reduction of $\omega $ to a symplectic form on $\mathcal{G} (C)/G$}. The lifted $G$-action on $\mathcal{G} (C)$ is obtained as the integration $\widetilde{\alpha }: G \times_M \mathcal{G} (C) \rightarrow \mathcal{G} (C)$ of the Lie algebroid morphism $\alpha $, see the proof of \cite[Thm. 3.6]{intinfact}. Since $\widetilde{\alpha }$ is also an action, $\Gamma:=G \times_M \mathcal{G} (C) \rightrightarrows \mathcal{G} (C) $ inherits an action groupoid structure, where $\widetilde{\alpha }$ is its target map and the projection $\text{pr}_2:G \times_M \mathcal{G} (C) \rightarrow \mathcal{G} (C)$ is its source. In order to prove that $\omega $ descends to a symplectic form on $\mathcal{G}(C)/G$, we have to check that the $G$-orbits are tangent to $\ker \omega $ and that $\mathtt{t}_{\Gamma }^* \omega =   \widetilde{\alpha }^* \omega =\mathtt{s}_{\Gamma }^* \omega = \text{pr}_2^* \omega $. 

{\em Step 4: the $G$-orbits are tangent to $\ker \omega $}. Infinitesimally, the $A$-action on $C$ lifts to $\mathcal{G} (C)$ as follows. Take $U\in \Gamma (A)$ of compact support. Since $\mathcal{L}_{\rho(U)}$ is a derivation of $C$ over $\rho(U)$, it integrates to a 1-parameter family of automorphisms $\psi_t$ of $C$. Lie's second theorem implies that $\psi_t$ lifts to a family of Lie groupoid automorphisms $\Psi_t$ of $\mathcal{G} (C)$. The infinitesimal generator of $\Psi_t$ is a multiplicative vector field $\widetilde{U}\in \mathfrak{X} (\mathcal{G} (C))$, i.e. a vector field which is a Lie groupoid morphism $\widetilde{U}:\mathcal{G} (C) \rightarrow T \mathcal{G} (C)$. Since $i_{\widetilde{U} }\omega :T \mathcal{G} (C) \rightarrow \mathbb{R} $ is a Lie groupoid morphism, in order to prove that $i_{\widetilde{U} }\omega=0$ we just have to check that its associated Lie algebroid morphism is zero. Let us denote by $\mathcal{B}$ the distribution tangent to the $A$-orbits, by construction, it is generated by the vector fields of the form $\widetilde{U}$.   

The Lie algebroid morphism $ \overline{\Lambda}_\mu: T C\oplus_{C} TC \rightarrow \mathbb{R}$ associated to $\omega: T \mathcal{G} (C_L)\oplus T \mathcal{G} (C) \rightarrow \mathbb{R}$ is defined by the linear 2-form $\Lambda_\mu=\mu^* \omega_{\text{can}}$, where $\omega_{\text{can}}$ is the canonical symplectic form on $T^*S$ \cite{burcab}. On the other hand, we have that $\widetilde{U} $ induces a Lie algebroid morphism $U'=C \rightarrow TC$ in the following way
\[ U'(\alpha_p )=T_p \alpha (U_p) - \overline{\left(\mathcal{L}_{\rho(U)} \alpha  \right)}_p \quad \forall p\in S, \]
where $\alpha \in \Gamma (C)$ and $\overline{\mathcal{L}_{\rho(U)} \alpha}_p$ is the vertical tangent vector to $C$ at $\alpha_p $ associated to $\left(\mathcal{L}_{\rho(U)} \alpha\right)_p$ \cite{intinfact}. So we have immediately that $i_{U'} \Lambda_{\mu}=0$ and hence $i_{\widetilde{U}} \omega =0$ which proves $\mathcal{B} \subset \ker \omega $.
   
On the other hand, we have that $\ker \omega_{\mathtt{t}(g)}\cap \ker T_{\mathtt{t}(g)}  \mathtt{s} =0$ since $\mu$ is injective and hence
\begin{align}  \ker \omega_{g}\cap \ker T_{g} \mathtt{s}\cong\ker \omega_{\mathtt{t}(g)}\cap \ker T_{\mathtt{t}(g)}  \mathtt{s} =0 \label{eq:ker2for} \end{align}  
for all $g\in \mathcal{G} (C)$, \cite[Lemma 3.1]{burint}. If $V\in \ker \omega_g$ and $W\in T_{\mathtt{s}(g)} G$, then for (any) $X\in T_g \mathcal{G} (C)$ composable with $W$ we have that
\[\omega_{\mathtt{s}(g)}(T \mathtt{s}(V),W)=\omega_g(V, T \mathtt{m}(W,X))-\omega_g(V,X)=0 \]   
	and so $T_g \mathtt{s}(v)\in \ker \omega_{\mathtt{s}(g)}$. But $T_g \mathtt{s}$ restricted to $\ker \omega_g$ is injective by \eqref{eq:ker2for}, so $\dim \ker \omega_g\leq \dim \ker \omega_{\mathtt{s}(g)} =\dim B=\dim \mathcal{B} $. Therefore, $\mathcal{B}=\ker \omega $. 

{\em Step 5: $\mathtt{t}_{\Gamma }^* \omega - \mathtt{s}_{\Gamma }^* \omega =0$}. We have to show that the linear 2-form corresponding to the multiplicative 2-form $\mathtt{t}_{\Gamma }^* \omega - \mathtt{s}_{\Gamma }^* \omega $ on $\Gamma $ vanishes \cite{burcab}. But this linear 2-form is nothing but $d(\alpha^* \lambda- \text{pr}_2^* \lambda)\in \Omega^2(G \times_M C)$ and we have that $\alpha^* \lambda- \text{pr}_2^* \lambda=0$. Therefore, $\omega $ is $\Gamma $-basic and it descends to a multiplicative symplectic form on the quotient $\mathcal{G} (C)/G$. 

Finally, if $(S/G,\sigma)$ is integrable, then its pullback Dirac structure $L$ along the projection map $S \rightarrow S/G$ is integrable \cite{burint}; $C$ can be identified with a Lie subalgebroid of $L$ and hence it is also integrable \cite{moeint}. \end{proof}
\begin{rema}\label{rem:source-simply-connected} A closer look at the $G$-action on $\mathcal{G} (C)$ reveals that it does not identify points on the same $\mathtt{s}$-fiber, therefore, $\mathcal{G} (C)/G$ is also source-simply-connected and so the fact that it is symplectic also follows from the integration of Lie bialgebroids in \cite{macxu2}. \end{rema}   
\subsubsection{Integrability of quotients of Poisson actions} A Lie group is a {\em Poisson group} if it is endowed with a Poisson structure such that the multiplication map is a Poisson morphism \cite{driham}. A {\em Poisson action} of a Poisson group on a Poisson manifold is a Lie group action which is a Poisson morphism \cite{semdres}.
\begin{coro}\label{crl:coired} Let $G$ be a Poisson group and suppose that there is a Poisson $G$-action on a Poisson manifold $S$. Suppose $G$ acts freely and properly on $S$. If $S$ is integrable, then the induced Poisson structure on the quotient $S/G$ is integrable. \end{coro}
\begin{proof} In this situation, $C$ as in Theorem \ref{thm:main} is the conormal bundle of the $G$-orbits on $S$ and it is a Lie subalgebroid of the cotangent Lie algebroid of $S$. Therefore, $C$ is integrable if so is $S$ and hence the result follows from Theorem \ref{thm:main}. \end{proof} 
\begin{rema} Notice that, in principle, we can apply Theorem \ref{thm:main} even when $S$ is not integrable by considering only the Lie subalgebroid $C \hookrightarrow T^*S$. In the case that $G$ is complete, this result appears in \cite{ferigl}. \end{rema}
Let us describe explicitly the integration of the quotient $S/G$ as in Corollary \ref{crl:coired} provided by Theorem \ref{thm:main}. Let $\Sigma(S) \rightrightarrows S$ be the source-simply-connected integration of $S$.  There is a Poisson map $\mu:\Sigma(S) \rightarrow G^*$ which is also a Lie groupoid morphism lifting the Lie algebroid morphism given the dual of the action map $T^*S \rightarrow \mathfrak{g}^*$. So there is an infinitesimal $\mathfrak{g} $-action on $\Sigma(S)$ which is not complete in general unless $G$ is complete \cite{ferigl}. Since the $\mathfrak{g} $-action is locally free, $\mu$ is a submersion and so $\mu^{-1}(1) \hookrightarrow \Sigma(S)$ is a Lie subgroupoid integrating the Lie subalgebroid $C \hookrightarrow T^*S$. If we consider $\mathcal{G} (C)$, Theorem \ref{thm:main} tells us that the $\mathfrak{g} $-action on $\mathcal{G} (C)$ given by the composition of the canonical morphism $\mathcal{G} (C) \rightarrow \mu^{-1}(1)$ with the inclusion $ \mu^{-1}(1) \hookrightarrow \Sigma(S)$ integrates to a $G$-action and that the quotient $\mathcal{G} (C)/G$ is a symplectic groupoid integrating $S/G$ (in fact, $\mathcal{G} (C)/G=\Sigma(S/G)$, see Remark \ref{rem:source-simply-connected}).

\subsubsection{Integrability of quotients of q-Poisson $G $-manifolds} If $(S,\pi)$ is a q-Poisson $\mathfrak{g}$-manifold, then there is a nonobvious but canonical Lie algebroid structure on $T^*S$, see \cite[Thm. 1]{liblasev}; we shall denote it by $(T^*S)_\mathfrak{g}$. It is immediate that $C$ as before is a Lie subalgebroid of $(T^*S)_\mathfrak{g}$ (provided the $\mathfrak{g}$-action is locally free). For the sake of making the analogy with the previous situation more evident, we shall also say in this case that $S$ is {\em integrable} if $(T^*S)_\mathfrak{g}$ is. If the $\mathfrak{g} $-action on $S$ is integrable to a $G$-action such that $\pi$ is invariant, $(S,\pi)$ is called a {\em q-Poisson $G$-manifold}.
\begin{coro} Let $(S,\pi)$ be a q-Poisson ${G}$-manifold. Then the Poisson structure induced on $ S/G$ is integrable if $S$ is integrable. \qed\end{coro}
 
Let us illustrate more precisely the integration provided of quotients of q-Poisson $G$-manifolds provided by Theorem \ref{thm:main} and let us compare it with the results of \cite{liblasev} about the integration of q-Poisson $G$-manifolds. 

Let $(S,\pi)$ be a q-Poisson $G $-manifold \cite{alemeikos}. Suppose that $G$ acts freely and properly on $S$. We have that \cite[Thm 1]{liblasev} states that the dual of the action map $T^*S \rightarrow \mathfrak{g}^*$ composed with the isomorphism $\mathfrak{g} \cong \mathfrak{g}^*$ induced by the bilinear form gives us a Lie algebroid morphism $(T^*S)_\mathfrak{g}\rightarrow \mathfrak{g} $. If $(T^*S)_\mathfrak{g}$ is integrable, then this morphism can be lifted to a moment map $\Phi:\mathcal{G}(T^*S)_ \mathfrak{g} \rightarrow G$ which makes $\mathcal{G}(T^*S)_ \mathfrak{g} $ into a q-Hamiltonian $\mathfrak{g}$-manifold (groupoid), see \cite[Thm 4]{liblasev}. If the $\mathfrak{g}$-action is locally free on $S$ (and hence on $\mathcal{G}(T^*S)_ \mathfrak{g}$), we have that $\Phi^{-1}(1)$ is a Lie groupoid. Indeed, \cite[Remark 3.3]{alemeimal} says that there is a 2-form $\varpi\in \Omega^2(\mathfrak{g} )$ such that, if we take a neighborhood $U$ of $1\in G$ covered diffeomorphically by a neighborhood of $0\in \mathfrak{g} $ using $\exp : \mathfrak{g} \rightarrow G$, then $\Omega:=\omega - \Phi^* \log^* \varpi$ is symplectic on $\Phi^{-1}(U)$ and $\mu:=\log \circ \Phi$ is a classical moment map for the $\mathfrak{g}$-action. Since the $\mathfrak{g} $-action is locally free on $S$, it is also locally free on $\mathcal{G}(T^*S)_ \mathfrak{g}$. The usual property of moment maps $(\ker T\mu )^\Omega=\mathfrak{g}_\mathcal{M} $ implies that $\mu$ is a submersion on $\Phi^{-1}(U)$ and then so is $\Phi$. If we consider the source-simply-connected integration $\widetilde{ \Phi^{-1}(1)}$ of $C=\text{Lie} (\Phi^{-1}(1))$, we have a $G$-action by automorphisms on $\widetilde{\Phi^{-1}(1)}$. Theorem \ref{thm:main} implies that $\widetilde{\Phi^{-1}(1)}/{G} $ is a symplectic groupoid which integrates the Poisson structure on $S/G$. 
\begin{exa}\label{intqpoigman} The most important examples of q-Poisson $G$-manifolds are the spaces of representations of fundamental groups of surfaces \cite{alemeimal,alemeikos}. For instance, the q-Poisson $G$-manifold associated to an annulus is $G$ itself and the integrability of the Lie algebroid structure on $(T^*G)_{\mathfrak{g} }$ is automatic, being an action Lie algebroid \cite{equger}. The difficulty in dealing with these spaces lies in the fact that the $G$-action on them is not free, see \cite{intquopoi}. \end{exa} 
  
\section{Double symplectic groupoids and gauge Poisson groupoids} In order to describe the following construction, we need to recall the following construction. A Lie subgroup $H\subset G$ of a Poisson group is coisotropic if it is coisotropic as a submanifold of $G$\footnote{Let $(M,\pi)$ be a Poisson manifold. A submanifold $C$ of $M$ is coisotropic if $\pi^\sharp(T^\circ C)\subset TC$, where $T^\circ C$ is the annihilator of $TC$. }; if $G$ is connected, this is equivalent to the annihilator $\mathfrak{h}^\circ\subset \mathfrak{g}^*$ being a Lie subalgebra. If a Poisson group acts in a Poisson fashion on a Poisson manifold, then the quotient of the manifold by a coisotropic subgroup is a Poisson manifold again, provided it is smooth \cite[Thm. 6]{semdres}. 

Let $(G,\pi_G)$ be a Poisson group acting in a Poisson fashion on a Poisson manifold $(M,\pi_M)$ If the $G$-action is free and proper, then $M/G$ inherits a unique Poisson structure such that the projection map is a Poisson morphism $M \rightarrow M/G$ \cite{semdres}. Now consider the Poisson manifold $M \times \overline{M}$, which is the product $M \times M$ endowed with the Poisson bivector $(\pi_M, -\pi_M)$. We have that the action of the Poisson group $G \times \overline{G}= (G \times G, (\pi_G,-\pi_G))$ on $M \times \overline{M}$ is Poisson again and the diagonal subgroup $G \hookrightarrow G \times \overline{G} $ is a coisotropic subgroup. Therefore, the quotient by the diagonal action $(M \times \overline{M} )/G $ is a Poisson manifold again. For a manifold $M$, the associated \emph{pair groupoid} is $M\times M \rightrightarrows M$ with source and target the projections on $M$, the multiplication is $\mathtt{m} ( (x,y), (y,z))=(x,z)$, the unit map $M\rightarrow M\times M$ is the diagonal inclusion and the inversion is given by $(x,y)\mapsto (y,x)$. Since $G$ acts by automorphisms on  $M \times M \rightrightarrows M$, the quotient 
\begin{align}  ((M \times \overline{M}) /G,\Pi) \rightrightarrows M/G \label{eq:gaugro} \end{align} 
is Lie groupoid again, called a {\em gauge groupoid} \cite{dufzun}. It turns out that the Poisson structure on this Lie groupoid is compatible with the groupoid structure in the following sense.
\begin{defi}[\cite{weicoi}] A {\em Poisson groupoid} is a Lie groupoid $\mathcal{G}  \rightrightarrows M$ with a Poisson structure on $\mathcal{G}$ such that the graph of the multiplication map is a coisotropic submanifold of $\mathcal{G}\times \mathcal{G}\times \overline{\mathcal{G}}$, where $\overline{\mathcal{G}}$ denotes $\mathcal{G}$ with the opposite Poisson structure. \end{defi} 
Poisson groups and symplectic groupoids (seen as Poisson manifolds) are extreme examples of Poisson groupoids. The pair groupoid $M \times \overline{M} \rightrightarrows M$ gives us another family of examples. Since the projection map $M \times \overline{M} \rightarrow ((M \times \overline{M})/G,\Pi)$ is a Lie groupoid morphism and a Poisson morphism, the bivector field $\Pi$ makes \eqref{eq:gaugro} into a Poisson groupoid, this fact was first observed by J.-H. Lu and her collaborators. Since we are dealing with Poisson groupoids, we can ask a more refined integrability question about this gauge Poisson groupoid. Some of the symplectic integrations of a Poisson groupoid may carry a {\em double symplectic groupoid} structure.  
\subsubsection{Double symplectic groupoids} \begin{defi}[\cite{browmac,macdou}] A groupoid object in the category of topological groupoids is called a double topological groupoid and it is denoted by a diagram of the following kind
\[ \xymatrix{ \mathcal{G}\ar@<-.5ex>[r] \ar@<.5ex>[r]\ar@<-.5ex>[d] \ar@<.5ex>[d]& H\ar@<-.5ex>[d] \ar@<.5ex>[d] \\ K\ar@<-.5ex>[r] \ar@<.5ex>[r] & S, }\]  
where each of the sides represents a groupoid structure and the structure maps of $\mathcal{G} $ over $H$ are groupoid morphisms with respect to $\mathcal{G} \rightrightarrows K$ and $H \rightrightarrows S$. A double topological groupoid as in the previous diagram is a {\em double Lie groupoid} if the following conditions are met: (1) each of the side groupoids is a smooth groupoid, (2) $H$ and $K $ are Lie groupoids over $S$, and (3) the double source map $(\mathtt{s}^H,\mathtt{s}^K):  \mathcal{G} \rightarrow {H} \times_S {K} $ is a surjective submersion (the superindices $\quad^H,\quad^K$ denote the groupoid structures $\mathcal{G} \rightrightarrows H$, $\mathcal{G} \rightrightarrows K$ respectively). \end{defi}

\begin{defi}[\cite{luwei}] A double Lie groupoid $\mathcal{G} $ with sides $K$ and $H$ over $S$ is a {\em double symplectic groupoid} if there is a symplectic structure on $\mathcal{G} $ making it into a symplectic groupoid over both $K$ and $H$. \end{defi}
In the previous definition, the Poisson structures induced on $K$ and on $H$ make them into Poisson groupoids over $S$, see \cite{macdou2}. If the Poisson structure on a Poisson groupoid is integrable by a double symplectic groupoid, we say that the Poisson groupoid is integrable.
\subsection{Integrability of gauge Poisson groupoids in the complete case} Now we shall see that a gauge Poisson groupoid as in \eqref{eq:gaugro} is integrable by a double symplectic groupoid if $G$ is complete.
\begin{thm}\label{thm:gaugro} Let $G$ be a complete Poisson group acting freely and properly on a Poisson manifold $M$. If $M$ is integrable, then the gauge Poisson groupoid $(M \times \overline{M})/G \rightrightarrows M/G$ is integrable by a double symplectic groupoid. \end{thm} 
\begin{proof} Let $G$ be a Poisson group acting freely and properly on a a Poisson manifold $M$ by a Poisson action and suppose that $M$ is integrable. Let $\Sigma(M) \rightrightarrows M$ be the source-simply-connected integration of $M$. Then we have a Lie groupoid morphism $\mu:\Sigma(M) \rightarrow G^*$ which is a moment map for a Poisson $\mathfrak{g} $-action on $\Sigma(M)$ \cite{burcabhoy,morla}. In that case, $(\mu,\mu): \Sigma (M) \times \overline{\Sigma(M)} \rightarrow G^* \times \overline{G^*}$ is also a Lie groupoid morphism and a moment map for a $\mathfrak{g} \times \mathfrak{g} $-action. Then $(\mu,\mu)^{-1}(G^*_\Delta)$ is a double Lie groupoid with sides $\Sigma(M)$ and $M \times \overline{M}$, where $G^*_\Delta \hookrightarrow G^* \times \overline{G^*}$ is the diagonal inclusion. If $G$ is complete, then the $\mathfrak{g} \times \mathfrak{g} $-action on $\Sigma (M) \times \overline{\Sigma(M) }$ is already integrable by a $G \times G$-action, see \cite[Thm. 1]{ferigl}. So we can consider the diagonal action restricted to $(\mu,\mu)^{-1}(G^*_\Delta)$. We have that $(\mu,\mu)^{-1}(G^*_\Delta)/G \rightrightarrows (M \times \overline{M})/G $ is a symplectic groupoid integrating the Poisson structure $\Pi$, see the proof of \cite[Thm. 3]{ferigl}. On the other hand, the diagonal $G$-action on $\Sigma (M) \times \overline{\Sigma(M)}$ clearly preserves the pair groupoid structure $\Sigma (M) \times \overline{\Sigma(M)} \rightrightarrows \Sigma(M)$, so we have a double Lie groupoid:
\[ \xymatrix{ (\mu,\mu)^{-1}(G^*_\Delta)/G \ar@<-.5ex>[r] \ar@<.5ex>[r]\ar@<-.5ex>[d] \ar@<.5ex>[d]& \Sigma(M)/G \ar@<-.5ex>[d] \ar@<.5ex>[d] \\ (M \times \overline{M})/G \ar@<-.5ex>[r] \ar@<.5ex>[r] & M/G. }\]
Finally, the symplectic structure on $(\mu,\mu)^{-1}(G^*_\Delta)/G$ is also multiplicative with respect to $(\mu,\mu)^{-1}(G^*_\Delta)/G \rightrightarrows \Sigma(M)/G$, since this is just a symplectic reduction of a pair symplectic groupoid. Therefore, we get a double symplectic groupoid over the gauge Poisson groupoid \eqref{eq:gaugro} as desired. \end{proof}
\begin{rema} We do not know if we can remove the completeness condition on $G$ in the previous theorem. If we knew that the $\mathfrak{g}\times \mathfrak{g} $-action on $\Sigma (M) \times \overline{\Sigma(M) }$ restricted to the diagonal $\mathfrak{g} $ integrates to a $G$-action on $(\mu,\mu)^{-1}(G^*_\Delta)$, then we could conclude that the quotient $(\mu,\mu)^{-1}(G^*_\Delta)/G$ is a double symplectic groupoid integrating the gauge Poisson groupoid \eqref{eq:gaugro}. Unfortunately, we cannot adapt Theorem \ref{main} to this situation since it only tells us how to produce $G$-actions by automorphisms and, in the complete case, the $G$-action on $(\mu,\mu)^{-1}(G^*_\Delta)$ is twisted by the moment map, see \cite{ferigl}. \end{rema}
\begin{rema} As a consequence of \cite[Proposition 7]{folpoigro}, we have that the $(\mu,\mu)^{-1}(G^*_\Delta)/G$-orbits of the units in $\Sigma(M)/G$ and in $(M \times \overline{M}) /G$ are symplectic groupoids themselves. So we get in this way a number of nontrivial examples of symplectic groupoids. Let us notice that the symplectic leaves of the units in $(M \times \overline{M}) /G \rightrightarrows M/G$ also give us symplectic groupoids even when $M/G$ is not integrable \cite[Proposition 12]{folpoigro}. \end{rema}
\begin{rema} The crucial fact in the proof of Theorem \ref{thm:gaugro} is that $(\mu,\mu): \Sigma(M) \times \Sigma(M) \rightarrow G^* \times G^*$ is a morphism of {\em double Poisson groupoids}, where $\Sigma(M) \times \overline{ \Sigma(M)}$ is seen as a double symplectic groupoid with sides $\Sigma(M)$ and $M \times \overline{M}$ and $G^* \times \overline{  G^*} \rightrightarrows G^* $ is seen as a Poisson 2-group \cite{poi2gro}. So we could formulate and prove a more general result about the integrability of quotients of Poisson groupoids by actions of coisotropic Lie 2-subgroups of Poisson 2-groups. Due to a certain lack of examples in this generality, we limit ourselves to the current formulation. \end{rema}  
The simplest examples of double symplectic groupoids associated to gauge Poisson groupoids are the following.
\begin{exa} Let $Q$ be a closed Poisson subgroup of a complete Poisson group $G$. Then the quotient by the diagonal action by left translations $(G \times \overline{G})/Q \rightrightarrows G/Q$ is a gauge Poisson groupoid. Since $G$ is complete, its source-simply-connected integration is the action groupoid $G^* \times G \rightrightarrows G$ associated to the dressing action $(u,x)\mapsto {}^ug$, where $G^*$ is the 1-connected integration of $\mathfrak{g}^*$. The lift of the $G$-action on $G$ by left translations to $G^* \times G$ is given by $a\cdot(u,b)=({}^au,a^ub)$, see \cite[Example 3.12]{ferigl}. The moment map $\mu:G^* \times G \rightarrow G^*$ is the projection on the first factor. Since $Q$ is a Poisson subgroup of $G$, the annihilator $\mathfrak{q}^\circ\subset \mathfrak{g}^*$ is an ideal and so there is a Lie group morphism $p:G^*\rightarrow Q^*$ integrating the projection $\mathfrak{g}^* \rightarrow \mathfrak{q}^*\cong \mathfrak{g}^*/ \mathfrak{q}^\circ$. Then the moment map for the lifted $Q$-action to $G^* \times G$ is $p\circ \mu:G^* \times  G \rightarrow Q^*$. As a consequence of Theorem \ref{thm:gaugro}, $(p\circ \mu,p\circ \mu)^{-1}(Q_\Delta^*)/Q$ is a double symplectic groupoid integrating the gauge Poisson groupoid $(G \times \overline{G} ) /Q \rightrightarrows G/Q$. 

More generally, we can do the following. Let $G$ be a Poisson group and let $Q \hookrightarrow G$ be a closed Poisson subgroup. Suppose that there is a (left) Poisson action of $Q$ on a Poisson manifold $Y$. Then the action $(G \times Y) \times( Q \times \overline{Q}  ) \rightarrow G \times Y$ given by $(g,y,a,b)\mapsto (ga,b^{-1}y)$ is a Poisson action, where $\overline{Q} $ denotes $Q$ with the opposite Poisson structure. Since the diagonal subgroup $Q_{\Delta} \hookrightarrow Q \times \overline{Q}$ is coisotropic, the associated bundle $G \times_Q Y:=(G \times Y)/Q$ is a Poisson manifold. For instance, if $Q_i\subset G$ is a closed Poisson subgroup for $i=1\dots n$, then the quotient $G \times_{Q_1} \times \dots \times_{Q_{n-1}} G/Q_n$ is integrable, see \cite{lumou} for a detailed description of this family of examples. In order to apply Theorem \ref{thm:gaugro}, we can take $M=G \times_{Q} \times \dots \times_{Q} G$ as before. Now consider the residual $Q$-action on the last factor. If $Q$ is complete, then Theorem \ref{thm:gaugro} implies that $(M \times \overline{M})/Q \rightrightarrows M/Q$ is integrable by a double symplectic groupoid, see \cite{lumou2} for Poisson groupoids related to this example. \end{exa}

\printbibliography
\end{document}